\documentclass{amsart}

\usepackage{amsfonts}
\usepackage{amsmath}
\usepackage{graphicx}
\usepackage{color}
\usepackage{tikz,subcaption}
\usepackage{biblatex}
\newtheorem{theorem}{Theorem}
\newtheorem{prop}{Proposition}

\newtheorem{lemma}{Lemma}
\newtheorem{defn}{Definition}
\newtheorem{cor}{Corollary}
\newtheorem*{gccc}{The Generalized Cosmetic Crossing Conjecture}
\newtheorem*{main}{Theorem 1}

\title{The $n$-adjacency graph for knots}

\author{Marion Campisi}
\address{Marion Campisi
\newline San Jos\'{e} State University
\newline San Jos\'{e}, CA
\newline USA}
\email{marion.campisi@sjsu.edu}
\author{Brandy Doleshal}
\address{Brandy Doleshal
\newline Sam Houston State University
\newline Huntsville, TX
\newline USA}
\email{bdoleshal@shsu.edu}
\author{Eric Staron}
\address{Eric Staron
\newline The University of Texas at Austin
\newline Austin, TX
\newline USA}
\email{estaron@math.utexas.edu}

\begin{document}

\begin{abstract}
A knot $K$ is called $n$-adjacent to a knot $K'$ if there is a set of $n$ crossing circles $\mathcal C$ in $K$ so that a generalized crossing change at any nonempty subset of crossings in $\mathcal C$ yields $K'$.
In this paper, the authors define a new graph $\Gamma_n$ to represent $n$-adjacency relationships between knots. We prove several results about this new object.
\end{abstract}

\maketitle

\section{Introduction}

A crossing change is a $\pm \frac{1}{2}$-Dehn surgery along a so called \emph{crossing circle}, i.e. a circle which bounds a \emph{crossing disk} which intersects a knot in two points, with algebraic intersection number zero. A generalized crossing change is a $\pm \frac{1}{2q}$-Dehn surgery along a crossing circle.
In this work, we say a knot $K$ is called $n$-adjacent to a knot $K'$ if there is a set of $n$ crossing circles $\mathcal C$ in $K$ so that a generalized crossing change at any nonempty subset of crossings in $\mathcal C$ yields $K'$. Where relevant, we make a distinction between $n$-adjacency resulting from crossing changes and $n$-adjacency resulting from generalized crossing changes by specifying that $q= 1$ in the former case.

Many researchers have studied $n$-adjacency in various contexts. Howards and Luecke \cite{HLntrivial} provide, when $q=1$, an upper bound for $n$ when a knot $K$ is $n$-adjacent to the unknot, in terms of the Seifert genus of $K$. Kalfagianni and Lin \cite{KalfLinFibering} \cite{KalfLinGenus} generalize these results when $g(K) > g(K')$ and use knot adjacency to obstruct fiberings of knots. Torisu \cite{TorisuTwoBridge2016} \cite{TorisuTwoBridge2008}  provides results about which knots are 2-adjacent to 2-bridge knots. Askitas and Kalfagianni \cite{AKOnKnotAdjacency} use the Alexander polynomial to obstruct fibered or alternating knots from being $n$-adjacent to the unknot for $n\ge 3$, when $q=1$. Tao \cite{TaoSymmetry} considers the relationship between 2-adjacency and knot polynomials when $q=1$, giving a condition under which the 2-adjacency relation is not symmetric. Most recently, Carney and Meike \cite{CM2AdjacentToUnknot} show that only 20 knots of 12 or fewer crossings are 2-adjacent to the unknot when $q=1$.

We contribute to this effort by defining a new object, called the $n$-adjacency graph, to represent the $n$-adjacency relation between knots, using generalized crossing changes. We prove several results about this object.

In Section \ref{section:background}, we provide the necessary background and definitions. In Section \ref{section:thegraph}, we define the $n$-adjacency graph for knots and prove several results about this new object, providing some connections to pre-existing literature. In Section \ref{section:2bridge}, we focus specifically on 2-adjacency of 2-bridge knots, and discuss a family of examples that prove the following:

\begin{theorem}\label{thm:main}
For every 2-bridge knot $K$, there are infinitely many 2-bridge knots $K'$ such that $K' \xrightarrow{\ 2\ } K$.
\end{theorem}

\section{Background and definitions}\label{section:background}

\begin{defn}
The knot $K$ is $n$-adjacent to the knot $K'$, denoted $K \xrightarrow{\ n\ } K'$ if there exist $n$ crossing circles, $C_1, C_2, \ldots C_n$, for $K$ such that a generalized crossing change on any nonempty subset of $\{C_1, C_2, \ldots C_n\}$ yields $K'$. 
\end{defn}

We call a crossing change \textit{order $d$} when it changes $d$ of the crossings in a knot diagram. Note that, in the $\pm \frac{1}{2q}$-Dehn surgery description of generalized crossing changes, we have $d = q$.

We denote an order $d$ generalized crossing change along crossing circle $C$ for $K$ by $K_{[C^d]}$.  Using this notation we see that a knot diagram $K$ is 2-adjacent to a knot diagram $K'$ if there exist crossing circles $C_1$ and $C_2$, with generalized crossing changes of orders $d_1$ and $d_2$ respectively  such that the following hold:
\begin{enumerate}
\item $K_{[C_1^{d_1}]}\cong K'$,
\item $K_{[C_2^{d_2}]}\cong K'$, and
\item $K_{[C_1^{d_1}, C_2^{d_2}]}\cong K'$.
\end{enumerate}

\begin{defn}
A crossing circle for $K$ that bounds a disk in the complement of $K$ is called a \textit{nugatory crossing circle}.
\end{defn}

\begin{defn}
A (generalized) crossing change on a knot $K$ and its corresponding crossing circle are called \textit{cosmetic} if the crossing change yields a knot isotopic to $K$ and the crossing change is performed on a non-nugatory crossing.
\end{defn}

The question is open as to whether there are knots that admit cosmetic crossing changes (Problem 1.58 of Kirby's problem list \cite{KirbyList}).

\begin{gccc} For any knot $K \subset S^3$, the only way to obtain $K$ from a generalized crossing change on a crossing $c$ is if $c$ is a nugatory crossing.
\end{gccc}

It follows from work of Scharlemann and Thompson \cite{ScharThompUnknot} that the unknot admits no cosmetic generalized crossing changes. Torisu \cite{TorisuNugCrossKnots} shows that the conjecture holds for 2-bridge knots, and Kalfagianni proves the same for fibered knots \cite{KalfFiberedKnots}. Ito \cite{ItoCosmeticPretzel} shows the conjecture is true for genus one pretzel knots. Balm and Kalfagianni \cite{BalmKalfCosmetic} prove the result for satellites, and Balm, Friedl, Kalfagianni and Powell \cite{BFKP} obstruct cosmetic crossing changes in genus 1 knots. Lidman and Moore \cite{LidMoo} prove the cosmetic crossing conjecture for other large collections of knots.

Here, we relate this conjecture to 2-adjacency.

In a 2-adjacency, we require that $K_{[C_1^{d_1}]}\cong K_{[C_1^{d_1}, C_2^{d_2}]}$ and $K_{[C_2^{d_2}]}\cong K_{[C_1^{d_1}, C_2^{d_2}]}$.  This implies that in $K_{[C_1^{d_1}]}$, $C_2$ either encompasses a nugatory crossing or is a cosmetic circle.  Likewise the crossing that intersects $C_1$ in $K_{[C_2^{d_2}]}$ is either nugatory or cosmetic. Similarly, an $n$-adjacency requires $K_{[C_i^{d_i}]}\cong K_{[C_i^{d_i}, C_j^{d_j}]}$ for all $1 \le i, j \le n$ and $i \neq j$. We note that, in general, if $K$ is in any class of knots for which the generalized cosmetic crossing conjecture is true,  $C_i$ must be nugatory in $K_{[C_j^{d_j}]}$ for all $1 \le i, j \le n$ and $i \neq j$. 

For this reason, we call a crossing circle $C_i \in \{C_1, \ldots, C_n\}$ \textit{trivializable} if the crossing circle is not nugatory in $K$ but becomes nugatory in $K'$ after a crossing change on some nonempty subset $S$ of $\{C_1, \ldots, C_n\}$ so that $C_i \notin S$. If a crossing circle is not nugatory in $K$ and is not trivializable for any such $S$, then the crossing circle is a cosmetic circle.

\section{The $n$-adjacency graph for knots}\label{section:thegraph}

In this section, we introduce a graph to represent the relationships obtained from $n$-adjacency on knots. 

Let $\Gamma_n$, for $n \ge 2$, be the graph defined in the following way. Each knot is represented by a vertex of $\Gamma_n$. A directed edge exists between two vertices $K$ and $K'$ if and only if $K \xrightarrow{\ n\ } K'$.  The edge is directed toward $K'$. 

If a knot $K$ is $n$-adjacent to itself, there is a bi-directed loop beginning and ending at $K$. We call $K$ a \textit{neighbor} of $K'$ when a directed edge exists pointing from $K$ to $K'$ and $K$ and $K'$ are not isotopic.

Note that we can define $\Gamma_1$ similarly, where a 1-adjacency is a generalized crossing change.

We further define $\Gamma_n^c$ as the subgraph of $\Gamma_n$ so that a directed edge exists between two vertices $K$ and $K'$ if and only if $K \xrightarrow{\ n\ } K'$ and some crossing circle in the $n$-adjacency is a cosmetic circle.

\begin{lemma}\label{lem:lonelyunkot}
The unknot is an isolated vertex in $\Gamma_n^c$ for $n\ge 2$.
\end{lemma}

\begin{proof}
Work of Scharlemann and Thompson \cite{ScharThompUnknot} indicates that the unknot has no cosmetic generalized crossing changes. This means that for any knot $K$ that is $n$-adjacent to the unknot, every crossing circle used in the $n$-adjacency must be trivializable, so we have no edges directed from $K$ to the unknot in $\Gamma_n^c$. 
\end{proof}

In fact, for the same reason, Lemma \ref{lem:lonelyunkot} is true of any knot that satisfies the Generalized Cosmetic Crossing Conjecture. 

\begin{theorem}
If the Generalized Cosmetic Crossing Conjecture holds for all knots, then $\Gamma_n^c$ is totally disconnected and has no loops.
\end{theorem}

\begin{proof}
If $K'$ is a knot that satisfies the Generalized Cosmetic Crossing Conjecture, there is no cosmetic circle for $K'$, and so there can be no knots $K$ such that $K \xrightarrow{\ n\ } K'$ with some cosmetic circle. Further $K$ has no cosmetic generalized crossing changes, so there can be no loops in $\Gamma_n^c$.
\end{proof}

\begin{lemma}\label{lem:nestedgraphs}
For $n \ge 1$, $\Gamma_{n+1} \subseteq \Gamma_{n}$.
\end{lemma}

It is easy to see from the definition of $n$-adjacency if $K \xrightarrow{n+1} K'$, then $K \xrightarrow{\ n \  } K'$ for $n\ge 1$. Then if $K$ is a neighbor of $K'$ in $\Gamma_{n+1}$, then $K$ is a neighbor of $K'$ in $\Gamma_n$.

\begin{theorem}\label{thm:gamma2wins}
For $n\ge 2$, $\Gamma_n \subseteq \Gamma_2$.
\end{theorem}

\begin{proof}
We induct on $n$, noting that the statement is tautologically true for $n=2$.

Suppose $\Gamma_n \subseteq \Gamma_2$. By Lemma \ref{lem:nestedgraphs}, $\Gamma_{n+1} \subseteq \Gamma_n$, so we have $\Gamma_{n+1} \subseteq \Gamma_n \subseteq \Gamma_2$.
\end{proof}

Because every adjacency graph is contained in $\Gamma_2$, studying  $\Gamma_2$ can give us valuable insight into adjacencies of higher order. On the other end of the spectrum, we define $\Gamma_\infty = \bigcap_{n = 2}^\infty \Gamma_n$.

We begin to connect the graph to the literature with the following theorem of 
Howards and Luecke \cite{HLntrivial}.

\begin{theorem}\label{thm:genus1}
Let $K$ be a non-trivial knot of genus $g$. Then $K$ fails to be $n$-adjacent to the unknot for all $n$ with $n \ge 3g-1$.
\end{theorem}

Further, Howards and Luecke \cite{HLntrivial} construct, for each $n$, a non-trivial knot that is $n$-adjacent to the unknot. This provides two results about $\Gamma_n$ that contrast with Lemma \ref{lem:lonelyunkot}.

\begin{cor}\label{cor:unknothasneighbors}
For every $n\ge 2$, the unknot has a non-trivial neighbor in $\Gamma_n$.
\end{cor}

\begin{prop}\label{prop:unknotinfvalance}
In $\Gamma_2$, the unknot has infinite valance. 
\end{prop}

\begin{proof}
By Corollary \ref{cor:unknothasneighbors}, the unknot has a non-trivial neighbor, $K_n$, in $\Gamma_n$ for every integer $n\ge 2$. Theorem \ref{thm:gamma2wins} tells us that $\Gamma_n \subseteq \Gamma_2$, so each $K_n$ is also a neighbor of the unknot in $\Gamma_2$. Let $\mathcal K$ be the collection of neighbors of the unknot in $\Gamma_2$.

Suppose, for the sake of contradiction, that the unknot, $\mathcal U$, has finitely many neighbors in $\Gamma_2$. Call the collection $\mathcal K = \{K_1, \ldots, K_m\}$ for some integer $m \ge 1$.  Theorem \ref{thm:genus1} tells us that for each  element $K_i\in \mathcal K$, $K_i$ is $n_i$-adjacent to $\mathcal U$ for some $n_i \ge 2$. Call $g_i$ the genus of $K_i$ for each $K_i \in \mathcal K$. Let $K_j \in \mathcal K$ be the knot in $\mathcal K$ with the largest genus, $g_j$. By Theorem \ref{thm:genus1}, $n_i < 3g_i-1$, and in particular, $n_j < 3g_j -1$. Since $g_j \ge g_i$ for $1 \le i \le m$, it follows that $n_i < 3g_j - 1$. That is, every $n_i$ is bounded above by the same finite number, and we have a contradiction to Corollary \ref{cor:unknothasneighbors}.
\end{proof}

We note here that because Howards and Luecke work in the case that $q=1$, the infinitely many neighbors for the unknot found in this proposition all result from order 1 crossing changes.

%Askitas and Kalfagiani \cite{AKOnKnotAdjacency} show that, when $q = 1$, if $K$ is $n$-adjacent to the unknot for some $n\ge 3$, then $K$ has trivial Alexander polynomial, and as a result, conclude that the only fibered or alternating knot that is $n$-adjacent to the unknot for $n\ge 3$ is the unknot itself. This gives us more information about the non-trivial neighbors of the unknot in $\Gamma_n$.

%\begin{cor}\label{cor:unknothasneighbors}
%For every $n\ge 3$, the unknot's non-trivial neighbors in $\Gamma_n$ that result are neither fibered nor alternating.
%\end{cor}

Kalfagianni and Lin \cite{KalfLinGenus} generalize the work in \cite{HLntrivial} with the following results, using generalized crossing changes.

\begin{theorem}
Suppose $K,K'$ are knots with $g(K)>g(K')$. If $K \xrightarrow{\ n\ } K'$, then $n\le 6g(K) - 3$.
\end{theorem}

\begin{cor}\label{cor:kalflincor}
If $K \xrightarrow{\ n\ } K'$ for all $n \in \mathbb N$, then $K$ and $K'$ are isotopic.
\end{cor}

Then we can prove the following about the adjacency graph.

\begin{prop}\label{prop:lonelyeveryone}
Every knot is isolated in $\Gamma_\infty$.
\end{prop}

\begin{proof}
Suppose on the contrary that $K$ is a neighbor of $K'$ in $\Gamma_\infty = \bigcap_{n=2}^\infty \Gamma_n$. Then $K \xrightarrow{\ n\ } K'$ for all $n\ge 2$. By Lemma \ref{lem:nestedgraphs}, we have that $\Gamma_2 \subseteq \Gamma_1$, so we have $K \xrightarrow{\ n\ } K'$ for all $n \in \mathbb N$. Then by Corollary \ref{cor:kalflincor}, $K$ and $K'$ are isotopic, a contradiction.  
\end{proof}

Additionally, Kalfagianni \cite{KalfFiberedKnots} proves the following theorem:

\begin{theorem}\label{thm:genus2}
If $K'$ is a fibered knot and $K \xrightarrow{\ n\ } K'$ for some $n\ge 2$, then either $K$ and $K'$ are isotopic or $g(K) > g(K')$.
\end{theorem}

Then we have the following result.

\begin{prop}
If $\mathcal U$ is the neighbor of a non-trivial knot $K'$ in $\Gamma_n$ for some $n\ge 2$, then $K'$ is not a fibered knot.
\end{prop}

\begin{proof}
Suppose on the contrary that the unknot, $\mathcal U$ is the neighbor of a fibered knot $K'$ in $\Gamma_n$ for some $n\ge 2$. By Theorem \ref{thm:genus2}, either $K'$ is trivial or $g(\mathcal U) > g(K')$. Both options produce a contradiction.
\end{proof}

\section{2-adjacency of 2-bridge knots}\label{section:2bridge}

In this section, we consider 2-adjacencies of 2-bridge knots in dialogue with questions of Torisu \cite{TorisuTwoBridge2008} and to contrast $\Gamma_2$ with $\Gamma_2^c$.

The usual way to consider a  2-bridge knot $S(p,q)$ is formed by twisting alternating strands by the list of numbers $a_1, -a_2, a_3, -a_4, \ldots, -a_{n-1}, a_n$. These numbers correspond to the numbers in a continued fraction which reduces to $\frac{p}{q}$. Here, we choose to think of a 2-bridge knot as a 3-string braid with an unusual closure, which we call the \textit{2-bridge closure}. 

\begin{defn}
Let $\beta$ be a braid whose word $w_\beta$ is expressed in the generators $\sigma_1$ and $\sigma_2$ of the braid group $B_3$ so that $\sigma_1$ is the first generator to appear in $w_\beta$. Further, let there be a zero-th vertical strand positioned to the left of the three strands in $\beta$ so that this zero-th strand does not interact with $\beta$. The \emph{2-bridge closure of $\beta$} has the following properties:
\begin{enumerate}
\item Strands 0 and 1 are connected at the top. 
\item Strands 2 and 3 are connected at the top.
\item If the final generator in $w_\beta$ is $\sigma_1$, strands 0 and 1 are connected at the bottom, and strands 2 and 3 are connected at the bottom. 
\item If the final generator in $w_\beta$ is $\sigma_2$, strands 0 and 3 are connected at the bottom, and strands 1 and 2 are connected at the bottom. 
\end{enumerate}
\end{defn}

In the braid group $B_3$, we write the braided part of the knot as the braid word $\beta=\sigma_1^{a_1}\sigma_2^{a_2} \sigma_1^{a_3}\cdots \sigma_2^{a_{n-1}}\sigma_1^{a_n}$.  Because of this  convention, we lose some of the conveniences of working in $B_n$, such as the ability to cyclically permute the generators. If the knot $K$ is the result of this closure on $\beta$ we say that $K$ \emph{corresponds} to $\beta$.  Note that it is possible for several distinct braids to correspond to the same knot type. 

The reader will notice that we have taken $n$ to be odd. When $n$ is even, the structure of a 2-bridge knot allows us to pull one crossing from the penultimate crossing to end on $\sigma_1$ as long as $|a_{n-1}|$ is at least 2. In the case that $a_{n-1}$ is equal to $\pm 1$, the structure of a 2-bridge knot allows us to add it to the penultimate crossing, increasing the number of crossings by 1.

\begin{defn}
We say the braid $\beta=\sigma_1^{a_1}\sigma_2^{a_2} \sigma_1^{a_3}\cdots \sigma_2^{a_{n-1}}\sigma_1^{a_n}$ has \emph{length} $n$.  
\end{defn}

\begin{defn}
A knot type $\mathcal{K}$ is said to be \emph{length $n$ 2-bridge} if there is a 2-bridge diagram $K$ of $\mathcal{K}$ with length $n$.  We also say $\mathcal{K}$ is of minimal length $n$.  
\end{defn}

As an example, the twist knot $T_k$ can be written as $\beta = \sigma_1^2 \sigma_2^{-k}$, so the twist knots are length-2 2-bridge knots.

Torisu \cite{TorisuNugCrossKnots} showed that the generalized crossing conjecture is true for two bridge knots, so we know that every 2-bridge knot is isolated in $\Gamma_2^c$. On the other hand, we have the following theorem.

\begin{main}
For every 2-bridge knot $K'$, there are infinitely many 2-bridge knots $K$ such that $K \xrightarrow{\ 2\ } K'$.
\end{main}

\begin{proof}
Call the knot shown in Figure \ref{fig:braid} by the name $K_\beta(m,n)$, where $m$ and $n$ are nonzero even integers. In terms of braid words in $B_3$, $K_\beta$ is the 2-bridge closure of the braid $\beta$, and $K_\beta(m,n)$ is represented by the word $\beta \sigma_2^m \beta^{-1} \sigma_2 ^n \beta$. Because we assume the 2-bridge length of $\beta$ is odd, we can assume that the word for $\beta$ ends in $\sigma_1$. 

We take the crossing circles $C_1$ and $C_2$ to each enclose one of the crossings in the box containing $m$ and $n$ crossings, respectively. Insisting that $\beta$ have odd 2-bridge length allows us to guarantee that the algebraic intersection number of $K_\beta(m,n)$ with the disk bounded by $C_i$ is zero. In terms of the braid word for $K_\beta(m,n)$, we can think of doing the crossing change on $C_1$ as deleting $\sigma_2^m$ from the braid word for $K_\beta(m,n)$. Then, both in terms of braid words and in terms of diagrams, we see that the generalized crossing change of order $m$ on $C_1$ results in the braid $\beta$. The same is true for the generalized crossing change of order $n$ on $C_2$.

If we do both generalized crossing changes, the word for $K_\beta(m,n)$ becomes $\beta$ as well, so we have a 2-adjacency from $K_\beta(m,n)$ to the 2-bridge closure of $\beta$.
\end{proof} 

\begin{center}
\begin{figure}
	\begin{tikzpicture}
		\draw[thick] (2,0) .. controls (2.25,-1) and (2.75, -1) .. (3, 0);
		\draw[thick] (2,9) .. controls (2.25,10) and (2.75, 10) .. (3, 9);
		\draw[thick] (-.5,0) .. controls (-.25,-1) and (.75,-1) .. (1,0);
		\draw[thick] (-.5,9) .. controls (-.25,10) and (.75,10) .. (1,9);
		\draw[thick] (-.5,0) -- (-.5,9);
		\draw[very thick] (0,0) rectangle (4,1);
		\draw[thick] (3,1) -- (3,2);
		\draw[thick] (2,1) -- (2,2);
		\draw[thick] (1,1) -- (1,4);
		\draw[very thick] (1.5,2) rectangle (3.5,3);
		\draw[thick] (2,3) -- (2,4);
		\draw[thick] (3,3) -- (3,4);
		\draw[very thick] (0,4) rectangle (4,5);
		\draw[thick] (1,5) -- (1,8);
		\draw[thick] (2,5) -- (2,6);
		\draw[thick] (3,5) -- (3,6);
		\draw[very thick] (1.5,6) rectangle (3.5,7);
		\draw[thick] (2,7) -- (2,8);
		\draw[thick] (3,7) -- (3,8);
		\draw[very thick] (0,8) rectangle (4,9);
		\node at (2, .5) {$\beta$}; 
		\node at (2, 4.55) {$\beta^{-1}$}; 
		\node at (2, 8.5) {$\beta$}; 
		\node at (2.5,6.5) {$m$}; 
		\node at (2.5,2.5) {$n$}; 
	\end{tikzpicture}\caption{$K'$}\label{fig:braid}
\end{figure}\hspace{6in}
\end{center}

Also, if there is a another braid that corresponds to $K$, then we have another infinite class of knots that are 2-adjacent to $K$. Additionally, Torisu \cite{TorisuTwoBridge2016} shows that there are infinitely many Montesinos knots that are 2-adjacent to any 2-bridge knot, so the 2-bridge knots have an abundance of neighbors in $\Gamma_2$. We conclude the following.

\begin{cor}\label{cor:geninfvalance}
Infinitely many vertices of $\Gamma_2$ have infinite valance.
\end{cor}

Using the construction from the proof of Theorem \ref{thm:main}, we have the following result.

\begin{prop}
There are infinitely many paths in $\Gamma_2$ that are arbitrarily long.
\end{prop}

We demonstrate by constructing infinitely many infinite paths in $\Gamma_2$.

\begin{proof}
Let $K_\beta(m,n)$ be the knot shown in Figure \ref{fig:braid}, which is $2$-adjacent to $K_\beta$. Let $\beta_1$ be the braid corresponding to $K_\beta(m,n)$ with the 2-bridge closure, as described above. Because $\beta$ has odd length, the length of $\beta_1$ is also odd, and we construct $K_{\beta_1}(m_1,n_1)$, with $m_1$ and $n_1$ nonzero even integers, in the same manner that $K_\beta(m,n)$ was constructed, by taking the braid to be $$(\beta \sigma_2^m \beta^{-1} \sigma_2^n \beta) \sigma_2^{m_1} (\beta \sigma_2^m \beta^{-1} \sigma_2 ^n \beta)^{-1} \sigma_2^{n_1} (\beta \sigma_2^m \beta^{-1} \sigma_2 ^n \beta) = \beta_1 \sigma_2^{m_1} \beta_1^{-1} \sigma_2^{n_1} \beta_1.$$
As in the proof of Theorem 7, $K_{\beta_1}$ is 2-adjacent to $K_{\beta}$ with crossing changes of order $m_1$ and $n_1$. Similary, we can construct, recursively, for every integer $i\ge 1$, $\beta_{i+1} = \beta_i \sigma_2^{m_i} \beta_i^{-1} \sigma_2^{n_i} \beta_i$, where $m_i$ and $n_i$ are nonzero even integers, so that $K_{\beta_{i+1}}$ is the 2-bridge closure of $\beta_{i+1}$ and $K_{\beta_{i+1}}(m_{i+1}, n_{i+1})$ is 2-adjacent to $K_{\beta_{i}}(m_{i}, n_{i})$. This construction provides an infinitely long list of knots, each of which is 2-adjacent to the previous in the list. Hence there exists an infinitely long path in $\Gamma_2$.

Because $\beta$ represents a 2-bridge knot, of which there are infinitely many, we can construct infinitely many such paths. Indeed for a single 2-bridge knot, we make a choice for the parameters $m_i$ and $n_i$ at every step, so each 2-bridge knot is the terminus of infinitely many paths of this type.
\end{proof}

\section{Questions}
We end our mathematical content with two questions about the $n$-adjacency graph.

In Proposition \ref{prop:unknotinfvalance} and Corollary \ref{cor:geninfvalance}, we find that the unknot and all of the 2-bridge knots are infinite valance in $\Gamma_2$. One then could naturally ask what other knots have infinite valance in $\Gamma_2$.

On the other hand, Proposition \ref{prop:lonelyeveryone} tells us that all knots are isolated in $\Gamma_\infty$. Kalfagianni and Lin \cite{KalfLinGenus} prove the following theorem

\begin{theorem}\label{thm:nbounded}
If $K$ and $K'$ are non-isotopic knots, then there is a constant $C(K,K')$ so that $K \xrightarrow{\ n\ } K'$ implies $n \le C(K,K')$.
\end{theorem}

As $n$ increases, the number of neighbors of $K'$ changes, and Lemma \ref{lem:nestedgraphs} tells us that this change is non-increasing. One could explore, for different classes of knots, at what rate the change in the number of neighbors occurs. 

\section{Acknowledgments}
The authors would like to thank Allison Moore for helpful comments.

This material is based upon work supported by the National Science Foundation under Grant No. DMS-1928930 and  the National Security Agency under Grant No. H98230-23-1-0004, while the authors participated in a program hosted by the Simons Laufer Mathematical Sciences Institute (formerly MSRI) in Berkeley, California, during the summer of 2023.

\printbibliography

\end{document}